\documentclass[a4paper,11pt]{article}

\usepackage{amsmath,amssymb,amsfonts}
\usepackage{amsthm}
\usepackage{hyperref}
\usepackage{verbatim}
\usepackage{xcolor}

\newcommand{\Nset}{\mathbb{N}}

\newcommand{\gmlu}{\mathrm{GMLU}}
\newcommand{\tp}{\mathrm{tp}}

\newcommand{\propmove}{$p$-move}
\newcommand{\lormove}{$\lor$-move}
\newcommand{\landmove}{$\land$-move}

\newcommand{\di}{\blacklozenge}
\newcommand{\bo}{\blacksquare}

\newcommand{\ddimove}{$\di^{\geq d}$-move}
\newcommand{\bbomove}{$\bo^{< d}$-move}
\newcommand{\edimove}{$\di^{= d}$-move}
\renewcommand{\AA}{\mathcal{A}}
\newcommand{\BB}{\mathcal{B}}
\newcommand{\MM}{\mathfrak{M}}
\newcommand{\G}{\mathcal{G}}
\newcommand{\mods}{\mathrm{Md}}
\newcommand{\size}{\mathrm{size}}

\newcommand*{\game}{\mathrm{GAME}}
\newcommand*{\depth}{\mathrm{depth}}
\newcommand*{\ebomove}{$\bo^{\neq d}$-move}

\theoremstyle{plain}
\newtheorem{theorem}{Theorem}[section]

\newtheorem{lemma}[theorem]{Lemma}
\newtheorem{corollary}[theorem]{Corollary}
\newtheorem{proposition}[theorem]{Proposition}
\theoremstyle{definition}

\newtheorem{remark}[theorem]{Remark}

\DeclareRobustCommand{\stirling}{\genfrac\{\}{0pt}{}}

\begin{document}

\title{A monotone connection between model 
class size and description length}

\author{Reijo Jaakkola\\ {\footnotesize Tampere University}\\ {\footnotesize Finland} \and Antti Kuusisto\\ {\footnotesize Tampere University}\\ {\footnotesize University of Helsinki}\\ {\footnotesize Finland}\\ \and Miikka Vilander\\ {\footnotesize Tampere University}\\ {\footnotesize Finland}}

\date{}


\maketitle
\thispagestyle{plain}
\pagestyle{plain}

\begin{abstract}
\noindent
This paper links sizes of model classes to the minimum lengths of their defining formulas, that is, to their description complexities. Limiting to models with a fixed domain of size n, we study description complexities with respect to the extension of propositional logic with the ability to count assignments. This logic, called GMLU, can alternatively be conceived as graded modal logic over Kripke models with the universal accessibility relation. While GMLU is expressively complete for defining multisets of assignments, we also investigate its fragments GMLU(d) that can count only up to the integer threshold d. We focus in particular on description complexities of equivalence classes of GMLU(d). We show that, in restriction to a poset of type realizations, the order of the equivalence classes based on size is identical to the order based on description complexities. This also demonstrates a monotone connection between Boltzmann entropies of model classes and description complexities. Furthermore, we characterize how the relation between domain size n and counting threshold d determines whether or not there exists a dominating class, which essentially means a model class with limit probability one. To obtain our results, we prove new estimates on r-associated Stirling numbers. As another crucial tool, we show that model classes split into two distinct cases in relation to their description complexity.
\end{abstract}

\section{Introduction}

This paper investigates how sizes of model classes are linked to the minimum lengths of formulas needed to define the classes. In the scenarios we consider, we first fix a class $\mathcal{M}$ of models that share a domain of the same finite size $n$. The model classes $M\subseteq \mathcal{M}$ are then studied with respect to the extension of propositional logic with the ability to count propositional assignments. We call this logic $\mathrm{GMLU}$, as it can alternatively be defined as graded modal logic over Kripke models with the universal relation. In order to obtain more fine grained results, we parameterize $\mathrm{GMLU}$ and study also its fragments $\mathrm{GMLU}_d$ that can count only up to the threshold $d\in \mathbb{Z}_+$. This also enables us to demonstrate how the relationship between minimum formula lengths and model class sizes develops when we gradually increase the expressive power of the logic used. For a model class $M\subseteq \mathcal{M}$, the \emph{description complexity} of $M$ with respect to $\mathrm{GMLU}_d$ is simply the minimum length of a formula of $\mathrm{GMLU}_d$ needed to define $M$, if such a formula exists.

In this paper we are particularly interested in the description complexities $C_d(M)$ of the logical equivalence classes $M$ determined by $\mathrm{GMLU}_d$ over $\mathcal{M}$. Let us write $\mathfrak{M}\equiv_d\mathfrak{M}'$ if the models $\mathfrak{M},\mathfrak{M}'\in \mathcal{M}$ satisfy the same set of formulas of $\mathrm{GMLU}_d$. Note that $C_d(M)$ of an equivalence class $M$ of $\equiv_d$ can also be regarded as the description complexity of each model $\mathfrak{M}\in M$, as the expressive power of $\mathrm{GMLU}_d$ suffices precisely to describe $\mathfrak{M}$ up to the equivalence $\equiv_d$. From this perspective, description complexity is analogous to Kolmogorov complexity. There exist well known links between Kolmogorov complexity and Shannon entropy, see for example \cite{vitanyi}. The recent work in \cite{stacspaper},\cite{stacsarxiv} demonstrates a way to conceive related results also in the scenario where \emph{relational structures are classified via logics}.  In particular, it is shown that the expected Boltzmann entropy of the equivalence classes of $\mathrm{GMLU}$ is asymptotically equivalent to the expected description complexity (with respect to $\mathrm{GMLU}$) times the size of the vocabulary considered. It is also shown that for $d=1$, the greatest equivalence class of $\mathrm{GMLU}_d$ has maximum description complexity among the classes. This paper builds on those results.

Firstly, as a crucial tool for our proofs, we establish a classification of description complexities into two distinct classes. This division is based on the numbers $n_i$ of elements realizing different propositional types $i\in I$ in models of the described model class; here $I$ is just an index set for the types. The division is then determined by whether or not $n_i = d$ for at least two different types. Using this, we establish a strong connection between model class sizes and description complexities. For each model $\mathfrak{M}\in \mathcal{M}$, let $\overline{n}_{\mathfrak{M}}$ denote the tuple $(n_i)_{i\in I}$ that gives the numbers $n_i$ of points realizing propositional
types in $\mathfrak{M}$. Furthermore, instead of recording numbers $n_i$ greater than the counting threshold $d$, simply put $d$ in $\overline{n}_{\mathfrak{M}}$. We define a poset $(\mathcal{M},\preceq_\tau)$ over the models, where $\tau$ is the vocabulary and the order $\preceq_\tau$ is based on comparing the tuples $\overline{n}_{\mathfrak{M}}$ coordinatewise. 
The order $\preceq_\tau$ is directly inherited also by the classes of $\equiv_d$ such that $M\preceq_\tau M'$ if and only if for some (or equivalently, all) models $\mathfrak{M}$ and $\mathfrak{M}'$ in the respective classes, we have $\mathfrak{M}\preceq_\tau \mathfrak{M}'$. We will prove that for all classes $M$ and $M'$ of $\equiv_d$ such that $M$ and $M'$ are $\preceq_\tau$-comparable, we have

%
%
%
\[ |M| < |M'| \ \Leftrightarrow \  C_d(M) < C_d(M').\]
In other words, over $\preceq_\tau$, the ordering of model classes according to size is \emph{identical} to the ordering based on description complexity. This is an intimate link between syntax and semantics. As a corollary, we obtain a corresponding relationship between Boltzmann entropies and description complexities of model classes.

We then investigate how the classes of $\equiv_d$ behave when we alter the domain size $n$ and counting threshold $d$. Note that increasing $d$ corresponds to moving to more and more expressive logics. First we observe that for thresholds $d$ and $d'>d$
and the corresponding Shannon entropies $H_S(\equiv_d)$ and $H_S(\equiv_{d'})$ of the model class distributions given by $\equiv_d$ and $\equiv_{d'}$, we have
\begin{align*}
& H_S(\equiv_d) < H_S(\equiv_{d'})\text{ when }d'\text{ is at most }n/2,\text{ and }\\
& H_S(\equiv_d) = H(\equiv_{d'})\text{ when } d\text{ is at least }n/2. 
\end{align*}
A similar result also follows for expected Boltzmann entropies
%
%
%
$H_B(\equiv_d)\text{ and}$ $H_B(\equiv_{d'}),$
but with the orders reversed, that is $H_B(\equiv_d) > H_B(\equiv_{d'})$ for $d'$ at most $n/2$. 
%
%
%
%
%
%

To get a better sense of the relative sizes of the classes when $n$ and $d$ are altered, we prove an asymptotic characterization of the class distributions as $n\rightarrow \infty$ and $d$ is a function of $n$.
Let us say that $\equiv_{d(n)}$ has a \emph{dominating class} if with limit probability one, a random model of size $n$ belongs to a maximum size class in $\equiv_{d(n)}$.
Similarly, all classes in $\equiv_{d(n)}$ are \emph{vanishing} if with limit probability zero, a random model of size $n$ belongs to a maximum size class.
Then the following results hold as $n \to \infty$. 
\begin{itemize}
\item
If $d(n) \leq n/2^{|\tau|} - f(n)$
where $f(n) = \omega(\sqrt{n})$,
then $\equiv_{{d(n)}}$ has a dominating class. 
\item
If $d(n) \geq n/2^{|\tau|} - f(n)$ where $f(n) = o(\sqrt{n})$, then $\equiv_{{d(n)}}$ has no dominating class.
\item
If $d(n) \geq n/2^{|\tau|} + f(n)$ where $f(n) = \omega(\sqrt{n})$, then every class in $\equiv_{d(n)}$ is vanishing.
\end{itemize}
%
%
%
One corollary of these results is that for $d(n) \leq n/t - f(n)$, if $f(n) = \omega(\sqrt{n})$, then with limit probability one, two random models of size $n$ cannot be separated in $\mathrm{GMLU}_{d(n)}$. Finally, we also give a non-asymptotic variant of the characterization of the class distributions for $\equiv_d$ including explicit bounds on $d$ for separating the cases where $\equiv_d$ has a \emph{majority class} or not. By a majority class, we mean a class containing more than half of all models in $\mathcal{M}$.

Concerning related work, as already mentioned, it is well known that entropy and Kolmogorov complexity are related. Indeed, for computable distributions, Shannon entropy links to Kolmogorov complexity to within a constant. This result is discussed, e.g., in \cite{vitanyi}, \cite{grunwald}, \cite{leung}. However, it is shown in \cite{teixeira} that the general link fails for R\'{e}nyi and Tsallis entropies.  See for example \cite{grunwald}, \cite{leung}, \cite{teixeira} for R\'{e}nyi and Tsallis entropies. The first connection between logical \emph{formula length} and entropy has---to our knowledge---been obtained in \cite{stacspaper}, \cite{stacsarxiv}, where expected Boltzmann entropy is shown to be asymptotically equivalent to description complexity.

Concerning further related work, we will next discuss the proof techniques used in the current paper. For proving bounds on formula sizes, we use \emph{formula size games} for the logics $\gmlu_d$. Indeed, variants of standard Ehrenfeucht-Fra\"{i}ss\'{e} games and (graded) bisimulation games would not suffice, as we need to deal with formula length, and thereby with all logical operators, including connectives. The formula size games for the logics $\gmlu_d$ will be developed below based on a similar game used in \cite{stacspaper}, \cite{stacsarxiv}. That game builds on the game for standard modal logic ML used and developed in \cite{HellaV19} for proving a nonelementary succinctness gap between first-order logic and ML. The first formula size game, developed by Razborov in \cite{Razborov90}, dealt with propositional logic. A later variant of the game was defined by Adler and Immerman for $\mathrm{CTL}$ in \cite{AdlerI03}. Designing the games for $\gmlu_d$ is relatively straightforward and based directly on similar earlier systems, but using them requires some nontrivial combinatorial arguments.


In addition to games, we also make use of a range of techniques for estimating model class sizes and description complexity. These include
Stirling's approximations and Chernoff bounds. In particular, to obtain our results, we prove new estimates on $r$-associated Stirling numbers, which may be of independent interest.

As a brief summary of our paper, the main objective is to elucidate the general picture of how description length relates to model class size. This also builds links between logic and notions of entropy.
The logic $\gmlu$ is suitable for the current study, and it even allows simple access to a chain of increasingly expressive logics $\gmlu_d$ via increasing $d$. The concluding section discusses possibilities for generalizing to further logics.
While the current paper focuses on theory, the notion of description complexity is also relevant in a range of applications. 
For example, in some currently active research on explainability in AI, minimal length specifications can be used as explanations of longer formulas. For work on this topic see, e.g., \cite{explainability2}, \cite{explainability1}.

The plan of the paper is as follows. After the preliminaries in Section \ref{preliminaries}, we prove crucial lower bounds for description complexity in Section \ref{descriptioncomplexitysection} using games. In Section \ref{monotoneconnection} we prove a monotone connection between model class size and description complexity, and in Section \ref{phase} we investigate phase transitions of class size distributions by varying $n$ and $d$. Section \ref{conclusion} concludes the paper.

\section{Preliminaries}\label{preliminaries}

We first define the logics studied in this work.
Let $\tau$ be a finite set of proposition symbols. We consider $\tau$ to be fixed throughout the entire paper.  The syntax of 
\textbf{graded universal modal logic} $\gmlu[\tau]$ is generated as follows (the syntactic choices will be explained later on):
\begin{align*}
\varphi := &\di^{\geq k} \psi \mid \bo^{< k} \psi \mid \di^{= k} \psi \mid \bo^{\neq k} \psi \mid \\
&\varphi \land \varphi \mid \varphi \lor \varphi \mid \di^{\geq k} \varphi \mid \bo^{< k} \varphi \mid \di^{= k} \varphi \mid \bo^{\neq k} \varphi \\
\psi := &p \mid \neg p \mid \psi \land \psi \mid \psi \lor \psi
\end{align*}
\noindent
Here $p \in \tau$ and $k \in \Nset$. Note that the formulas of $\gmlu[\tau]$ have proposition symbols only in the scope of modal operators. Furthermore, all formulas are given in negation normal form. In the current paper, $\neg\varphi$ will always mean a formula where $\neg$ has been pushed all the way to the level literals.

Let $\MM$ be a Kripke model with domain $W$. 
In this paper, modal logics will always have a unary vocabulary, so therefore Kripke models will not be associated with a binary accessibility relation.
We define the semantics of the global graded modalities as follows:
$(\MM, w) \vDash \di^{\geq k}\varphi \Leftrightarrow$ there exist at least $d$ elements $v \in W$ such that $(\MM, v) \vDash \varphi$ and $(\MM, w) \vDash \di^{= k} \varphi \Leftrightarrow $ there exist exactly $d$ elements $v \in W$ such that $(\MM, v) \vDash \varphi$.
Additionally, $(\MM, w) \vDash \bo^{< k} \varphi \Leftrightarrow (\MM, w) \vDash \neg \di^{\geq k} \neg \varphi$ and $(\MM, w) \vDash \bo^{\neq k} \varphi \Leftrightarrow (\MM, w) \vDash \neg \di^{= k} \neg \varphi$. 
The semantics of the Boolean connectives $\neg,\wedge,\vee$ is defined in the usual way. Notice
that $\di^{\geq k}$ and $\bo^{< k}$ as well as $\di^{= k}$ and $\bo^{\neq k}$ are
dual to each other. Thus the modalities of $\gmlu$ are the diamonds $\di^{\geq k}$,
$\di^{= k}$ and their duals $\bo^{< k}$, $\bo^{\neq k}$.
Intuitively $\bo^{< k}$ (respectively, $\bo^{\neq k}$) means that all points satisfy $\varphi$, except for
some number $m < k$ (resp. $m\not=k$) of exceptions. 


Let $\MM$ be a Kripke model over $\tau$ and $\varphi \in \gmlu[\tau]$. The point-free truth relation is defined such that
$\MM \vDash \varphi$ if and only if $\MM,w \vDash \varphi
\text{ for all }w\in W$. As all proposition symbols occur in the scope of a global modality, $\MM \vDash \varphi$ if and only if there exists some $w\in W$ such that $\MM,w \vDash \varphi$.
Clearly truth of $\gmlu$-formulas does not depend on the evaluation point $w$. This independence property is the reason behind the definition of the syntax of $\gmlu$ such that proposition symbols are guaranteed to be in the scope of modal operators. For a set $M$ of pointed models, we denote $M \vDash \varphi \Leftrightarrow (\MM, w) \vDash \varphi$ for every $(\MM, w) \in M$. 

A \textbf{propositional type} $\pi$ over $\tau$ is a maximally consistent set of literals (that is, proposition symbols and negated proposition symbols). Therefore $\pi$ has exactly one of $p, \neg p$ for each symbol $p\in\tau$. We henceforth refer to propositional types as just types.
The number of types over $\tau$ is
denoted by $t = 2^{|\tau|}$.

The \textbf{counting depth} of a formula $\varphi \in \gmlu[\tau]$, denoted $\depth(\varphi)$, is defined as follows:
\begin{itemize}
    \item $\depth(\alpha) = 0$ for any literal $\alpha$,
    \item $\depth(\varphi \wedge \psi) = \depth(\varphi \vee \psi) \\ = \max(\depth(\varphi), \depth(\psi))$,
    \item $\depth(\di^{\geq k} \varphi) = \depth(\bo^{< k} \varphi) = k$,
    \item $\depth(\di^{= k} \varphi) = \depth(\bo^{\neq k} \varphi) = k+1$.
\end{itemize}
We denote by $\gmlu_d[\tau]$ \textbf{the counting depth $d$ fragment of $\gmlu[\tau]$}, where the counting depth of formulas is restricted to at most $d$. The results of this paper are formulated for the logics $\gmlu_d[\tau]$. 
Note that $\depth(\di^{= k} \varphi) = k+1$ while $\depth(\di^{\geq k} \varphi) = k$. We give some intuition to explain this choice. First of all, $\di^{= d-1} \varphi \equiv \di^{\geq d-1} \varphi \land \neg \di^{\geq d} \varphi$, so we see that when the counting depth of formulas is restricted to $d$, the allowed ``exact counting'' diamonds $\di^{=k}$ always have $k\leq d - 1$ and thus they add no expressive power over ``threshold counting'' diamonds $\di^{\geq k}$ with $k\leq d$.
Moreover, $k+1$ also corresponds to the number of quantifiers required to express ``exact counting'' of $k$ elements in monadic first-order logic.

The \textbf{size} of a formula $\varphi \in \gmlu[\tau]$, denoted $\size(\varphi)$, is defined as follows:
\begin{itemize}
    \item $\size(\alpha) = 1$ for any literal $\alpha$,
    \item $\size(\varphi \wedge \psi) = \size(\varphi \vee \psi) = \size(\varphi) + \size(\psi) + 1$,
    \item $\size(\di^{\geq k} \varphi) = \size(\bo^{< k} \varphi) = \size(\varphi) + k$,
    \item $\size(\di^{= k} \varphi) = \size(\bo^{\neq k} \varphi) = \size(\phi) + k + 1$.
\end{itemize}
Note that \emph{all literals} have the same size. This is because we wish to consider negative (that is, negated)
information and positive (that is, non-negated)
information as \emph{equal} in relation to formula size.
This idea explains why we defined $\gmlu$ so that formulas are in negation normal form.


Let $\mathcal{M}$ be the set of all $\tau$-models with the fixed domain $W = \{1, \dots, n\}$. When $\mathcal{M}$ is clear from the context, a formula $\varphi \in \gmlu_d$ is said to \textbf{define} a set $M \subseteq \mathcal{M}$ if for every $\mathfrak{M} \in \mathcal{M}$, we have $\mathfrak{M} \vDash \varphi$ if and only if $\mathfrak{M} \in M$. The set $M$ is then called $\gmlu_d$-\textbf{definable}. The $\gmlu_d$-\textbf{description complexity} $C(M)$ of a $\gmlu_d$-definable set $M$ is the minimum size of a formula $\varphi \in \gmlu_d$ which defines $M$.


We may write $\mathfrak{M}\equiv_d\mathfrak{N}$ if the models $\mathfrak{M}$ and $\mathfrak{N}$ satisfy exactly the same $\gmlu_d$-formulas.
The relation $\equiv_d$ is clearly an equivalence relation and defines a natural related partition. 
For an example of description complexity, consider $\gmlu_1[\tau]$ for the case with the singleton alphabet $\tau =\{p\}$. The model where $p$ is true in every point constitutes a singleton class in the partition of models defined by $\equiv_1$. The description complexity of this class is 2 as a minimum size formula that defines the class is $\bo^{<1} p$.

 Let $I := \{1,\dots,t\}$ and fix an enumeration $(\pi_i)_{i \in I}$ of all the types over the propositional vocabulary $\tau$. Now let $M$ be an equivalence class of $\equiv_d$ over the set $\mathcal{M}$ of models of size $n$ with domain $W = \{1,\dots, n\}$. For a type $\pi_i$, all models in the class either have exactly $n_i$ points of type $\pi_i$ for some $n_i < d$, or all the models in $M$ have at least $d$ points of type $\pi_i$. In the latter case we define $n_i := d$. We thus get a characterization of the classes of $\equiv_d$ in terms of $t$-tuples $\overline{n}$. 
 A $t$-tuple $\overline{n} = (n_i)_{i \in I}$ is called $(n,d)$-\textbf{admissible}, if $n_i \leq d$ for every $i \in I$, $\sum_{i \in I} n_i \leq n$ and either there is at least one $i \in I$ such that $n_i = d$ or $\sum_{i \in I} n_i = n$. 
 Note that $(n,d)$-admissible tuples $\overline{n}$ and equivalence classes of $\equiv_d$ are in one-to-one correspondence. Thus we can write $M_{\overline{n}}$ for the class corresponding to $\overline{n}$.

Given an $(n, d)$-admissible tuple $\overline{n}$ and a type $\pi_i$ that has precisely the same number $k$ of realizing points in every model $\MM \in M_{\overline{n}}$, we denote this number $k$ by $|\pi_i|_{\overline{n}}$. (Note that $k$ can be greater than $d$.)
We will often omit the tuple ${\overline{n}}$ in the subscript of $|\pi_i|_{\overline{n}}$ when it is clear from the context. 

Following \cite{stacsarxiv}, we define the \textbf{Boltzmann entropy of a class} $M$ as $H_B(M) := \log(|M|)$. As discussed in \cite{stacsarxiv}, this terminology comes from statistical mechanics, where Boltzmann entropy measures the randomness of a \emph{macrostate} via the number of \emph{microstates} that correspond to it. The idea is that a larger macrostate is ``more random'' (or ``less specific'') since it is more likely to be hit by a random selection. In statistical mechanics, the formula for Boltzmann entropy is $k_B\ln\Omega$, where $\Omega$ is the number of microstates and $k_B$ the Boltzmann constant. In our definition, we use the binary logarithm (and do not use $k_B$). As a general intuition, it is natural to associate a formula $\varphi$ (or the class it defines) with a macrostate, while the models of $\varphi$ are then the corresponding microstates.

Consider now the following natural probability distribution over the equivalence classes of $\equiv_d$: $p_{\equiv_d}(M) = |M|/|\mathcal{M}|$ for each class $M$. We again refer to this distribution with the symbol $\equiv_d$ (with slight abuse of notation). We define the \textbf{Boltzmann entropy of the distribution $\equiv_d$} as the expected value of $H_B$ over the distribution $\equiv_d$ and we denote it by $H_B(\equiv_d)$. In other words, we define $H_B(\equiv_d)$ as $\sum_{M \in \mathcal{M} / \equiv_d} p_{\equiv_d}(M) H_B(M)$. Roughly speaking, $H_B(\equiv_d)$ is large when $\equiv_d$ is far from the uniform distribution.

The Boltzmann entropy of the distribution $\equiv_d$ is closely related to the \textbf{Shannon entropy} $H_S(\equiv_d)$ of $\equiv_d$, which we define as the expected value of $-\log(p_{\equiv_d}(M))$ over the distribution $\equiv_d$. More explicitly, we define $H_S(\equiv_d)$ as $-\sum_{M \in \mathcal{M} / \equiv_d} p_{\equiv_d}(M) \log(p_{\equiv_d}(M))$. In contrast to Boltzmann entropy, Shannon entropy measures randomness of $\equiv_d$ by looking at how uniform the distribution is. Indeed, if $\equiv_d$ contains a very large class, then its Shannon entropy is small, while its Boltzmann entropy is relatively large. The following result, which was shown in \cite{stacsarxiv} in a more general setting (with slightly different notation), formally establishes that the two notions of entropy are complementary in nature.


\begin{proposition}\label{prop:shannon_boltzmann}
   $H_S(\equiv_d) + H_B(\equiv_d) = |\tau|n$
\end{proposition}



\section{Description complexity}\label{descriptioncomplexitysection} 

In this section we investigate the $\gmlu_d$-description complexity of equivalence classes in the partition $\equiv_{d}$. We will utilize a formula size game for $\gmlu_d$.

Let $I = \{1, \dots, t\}$ and let $(\pi_i)_{i \in I}$ be an enumeration of the types of the set $\tau$ of proposition symbols. Let $d \leq n$ be the counting depth. We consider the partition induced by $\gmlu_d$ for models of size $n$. For each admissible tuple $\overline{n} = (n_i)_{i \in I}$ there is an equivalence class $M_{\overline{n}}$, where each type $\pi_i$ realized $n_i$ times, with $n_i = d$ meaning the type $\pi_i$ is realized \emph{at least} $d$ times. We denote the number of occurrences of $d$ in $\overline{n}$ by $k_d$. 

\medskip

Such a class $M_{\overline{n}}$ can be defined via the following $\gmlu_d$ formula:
\[
\varphi(\overline{n}) := \bigwedge\limits_{n_i < d} \di^{= n_i} \psi(\pi_i) \land \bigwedge\limits_{n_i = d} \di^{\geq d} \psi(\pi_i)
\]
The size of the formula $\varphi(\overline{n})$ is $\sum_{i \in I} n_i + t(2|\tau|+1)-k_d-1$.

If $n_i = d$ for at most one $i \in I$, then $M_{\overline{n}}$ can be defined by a smaller formula. Let $\pi_j$ be a type with maximal $n_j$ in $\overline{n}$. Now $M_{\overline{n}}$ is also defined by the formula
\[
\varphi'(\overline{n}) := \bigwedge\limits_{i \neq j} \di^{= n_i} \psi(\pi_i).
\]
To see this, recall that we restrict to models of size $n$. As all other types have been exactly specified, the only option for the remaining points is the type $\pi_j$. The size of $\varphi'(\overline{n})$ is $n - |\pi_j| +(t-1)(2|\tau|+1)-2$, where $|\pi_j|$ denotes the number of points in models of $M_{\overline{n}}$ with type $\pi_j$.

We define the constant $c_{\tau} := t(2|\tau|+1)-1$. By the formulas above we see that $C(M_{\overline{n}}) \leq \sum_{i \in I} n_i + c_{\tau}$ if $k_d \geq 2$ and $C(M_{\overline{n}}) \leq n - |\pi_j| + c_{\tau}$ if $k_d \leq 1$. We will use the formula size game for $\gmlu_d$ to show that these bounds are optimal up to the constant $c_{\tau}$.

Let $r_0 \in \Nset$ and let $\AA_0, \BB_0$ be sets of $\tau$-models. The $\gmlu_d$-formula size game $\game_d(r_0, \AA_0, \BB_0)$ has two players, S and D. Positions of the game are of the form $P = (r, \AA, \BB)$ and the starting position is $P_0 = (r_0, \AA_0, \BB_0)$. In a position $P$, if $r = 0$, then D wins. Otherwise S chooses between the following moves:

\smallskip

\noindent \textbf{\propmove}: S chooses a $\tau$-literal $\alpha$. The game ends. If $\AA \vDash \alpha$ and $\BB \vDash \neg \alpha$, then S wins. Otherwise D wins. S cannot make this move if he has not made a modal move so far.

\smallskip

\noindent \textbf{\lormove}: S chooses $\AA_1, \AA_2 \subseteq \AA$ such that $\AA_1 \cup \AA_2 = \AA$ and $r_1, r_2 \geq 1$ such that $r_1 + r_2 + 1 = r$. D chooses whether the next position is $(r_1, \AA_1, \BB)$ or $(r_2, \AA_2, \BB)$.

\smallskip

\noindent \textbf{\landmove}: The same as a \lormove\ with the roles of $\AA$ and $\BB$ switched.

\smallskip

\noindent \textbf{\ddimove}: S chooses a number $k \in \Nset$ with $k \leq d$ and $k < r$. For every $(\MM, w) \in \AA$, S chooses $k$ different points $v \in W$. Let $\AA'$ be the set of models $(\MM, v)$ chosen this way. For every $(\MM, w) \in \BB$, S chooses $n-k+1$ different points $v \in W$. Let $\BB'$ again be the set of models chosen. The next position of the game is $(r-k, \AA', \BB')$.

\smallskip

\noindent \textbf{\bbomove}: The same as a \ddimove\ with the roles of $\AA$ and $\BB$ switched.

\smallskip

\noindent \textbf{\edimove}: S chooses a number $k \in \Nset$ with $k < d$ and $k < r$. For every $(\MM, w) \in \AA$, S chooses a set $P_{\MM, w}$ of $k$ different points. Let $N_{\MM, w} := W \setminus P_{\MM, w}$. 
For every $(\MM, w) \in \BB$, S chooses either a set $P_{\MM, w}$ of $k+1$ different points or a set $N_{\MM, w}$ of $|\MM|-k+1$ different points. Finally we let $\AA' := \{(\MM, v) \mid (\MM, w) \in \AA \cup \BB, v \in P_{\MM, w}\}$ and $\BB' := \{(\MM, v) \mid (\MM, w) \in \AA \cup \BB, v \in N_{\MM, w}\}$. The next position of the game is $(r-k-1, \AA', \BB')$.

\smallskip

\noindent \textbf{\ebomove}: The same as a \edimove\ with the roles of $\AA$ and $\BB$ switched.

\medskip

The formula size game characterizes the size of formulas that separate model classes. This is formalized in the following theorem:
\begin{theorem}\label{thm:game}
The following statements are equivalent:
   \begin{enumerate}
       \item S has a winning strategy in the game $\game_d(r, \AA, \BB)$.
       \item There is $\varphi \in \gmlu_d[\tau]$ with size at most $r$ such that $\AA \vDash \varphi$ and $\BB \vDash \neg\varphi$.
   \end{enumerate}
\end{theorem}
\begin{proof}
    Easy proof by induction. See \cite{HellaV19} for a version of the proof for basic modal logic.
\end{proof}

It will be useful for the proofs below to note that if (essentially) the same model is on both sides of the game, then D has an easy winning strategy.
\begin{lemma}\label{lem:samemodel}
Let $P = (r, \AA, \BB)$ be a position of a game $\game_d(r_0, \AA_0, \BB_0)$. Let there be propositionally equivalent versions $(\MM, w) \in \AA$ and $(\MM, v) \in \BB$ of the same model $\MM$. Now D has a winning strategy from position $P$.
\end{lemma}
\begin{proof}
It is easy to see that the pair of propositionally equivalent versions of the same model is maintained through any modal move of S. For \lormove s and \landmove s one of the two possible following positions will always have such a pair of models and the strategy of D is to choose this position.
\end{proof}

\medskip

 Let $\overline{n}$ be an admissible tuple and assume that $n_1$ is one of the largest coordinates. We first define the sets $\AA_{\overline{n}}$ and $\BB_{\overline{n}}$ of models for the game. All models have the same universe $W = \{1, \dots, n\}$. We denote  $\mathrm{supp}(\overline{n}) := \{i \in I \mid n_i > 0\}$. We first define a model $\MM_0$ as follows. For each $i \in I$, $i \neq 1$, the type $\pi_i$ is realized precisely $n_i$ times. The type $\pi_1$ is realized $n- \sum_{i \neq 1} n_i$ times. Additionally, the point 1 is of type $\pi_1$. Intuitively the model $\MM_0$ is a model of the class $M_{\overline{n}}$, where all points not fixed by the tuple $\overline{n}$ are of type $\pi_1$.
We set $\AA_{\overline{n}} := \{(\MM_0, 1)\}$. 

\medskip

For the set $\BB_{\overline{n}}$ we define models $\MM_{i \rightarrow j}$ for some pairs $(i, j) \in \mathrm{supp}(\overline{n}) \times \mathrm{supp}(\overline{n})$ as follows. 
For $i \neq 1$ and any $j \in \mathrm{supp}(\overline{n})$, the model $\MM_{i \rightarrow j}$ is obtained from the model $\MM_0$ by changing one point of type $\pi_i$ to type $\pi_j$. If $i = 1$ and $n_j = d$, then the model $\MM_{i \rightarrow j}$ is obtained from $\MM_0$ by changing $|\pi_1|_{\MM_0}-d+1$ points $w \neq 1$ of type $\pi_1$ to type $\pi_j$. 
If $i = 1$ and $n_j < d$, no model is defined for the pair $(i, j)$. We set $\BB_{\overline{n}} := \{(\MM_{i \rightarrow j}, 1) \mid i, j \in \mathrm{supp}(\overline{n}) \}$.

\medskip

In terms of their tuples $\overline{n}'$, the models $\MM_{i \rightarrow j}$ have $n'_i = n_i -1$ and $n'_j = n_j +1$, except if $n_j = d$, in which case $n'_j = n_j = d$. For all other indices $\ell$, $n'_\ell = n_\ell$. All models $(\MM_0,1)$ and $(\MM_{i \rightarrow j},1)$ are propositionally equivalent as they realize the type $\pi_1$ in the point 1.

\medskip

Let $P = (r, \AA, \BB)$ be a position of the game $\game_d(r, \AA_{\overline{n}},\BB_{\overline{n}})$. We define a directed graph $\G(\AA, \BB) := (V, E)$, where $V = \mathrm{supp}(\overline{n})$ and $(i, j) \in E$ if there are propositionally equivalent $(\MM_0, w) \in \AA$ and $(\MM_{i \rightarrow j}, v) \in \BB$, or vice versa with respect to $\AA$ and $\BB$. We call a set $C \subseteq \mathrm{supp}(\overline{n})$ a \textbf{cover} of $\G(\AA, \BB)$ if for every $(i, j) \in E$, we have that either $i \in C$ or both $j \in C$ and $n_j < d$. The \textbf{cost} $r(C)$ of a cover $C$ is 
\[
r(C) := \sum\limits_{i \in C} n_i.
\]
Intuitively, the definition of a cover means that including an index $i \in I$ generally covers all edges to and from $i$. The exception is that when $n_i = d$, incoming edges are not covered.

In the proof of the following lemma we use the notation $\mods(\AA) = \{\MM \mid (\MM, w) \in \AA, w \in W\}$ for the set of models that have at least one pointed version in the set $\AA$ of pointed models. We also denote by $\tp(A)$ the set of propositional types realized in the set $A$ of points in a model $\MM$ that will below be clear from the set of points.

\begin{lemma}\label{coverlemma}
Let $P = (r, \AA, \BB)$ be a position of the game $\game_d(r, \AA_{\overline{n}},\BB_{\overline{n}})$ and let 
\[
R(P) := \min \{r(C) \mid C \text{ \emph{is a cover of} } \G(\AA, \BB)\}.
\]
If $r < R(P)$, then D has a winning strategy from position $P$.
\end{lemma}
\begin{proof}
We show for all possible moves of S that either the condition $r < R(P)$ is maintained or D has a winning strategy for some other reason. Since the definition of the graph $\G(\AA, \BB)$ is symmetrical with respect to $\AA$ and $\BB$, we only need to handle one of each pair of dual moves. 

\smallskip

\noindent \textbf{\propmove}: Since $0 < r < R(P)$, we have propositionally equivalent models on both sides of the game and thus D wins if S makes any \propmove.

\smallskip

\noindent \textbf{\lormove}: Let $\AA_1, \AA_2 \subseteq \AA$ and $r_1, r_2 \geq 1$ be the choices of S and let $P_1 = (r_1, \AA_1, \BB)$ and $P_2 = (r_2, \AA_2, \BB)$. For each $(i, j) \in E$ there is a pair of propositionally equivalent models $(\MM, w) \in \AA$ and $(\MM', v) \in \BB$ as witnesses. Since $\AA_1 \cup \AA_2 = \AA$, each model $(\MM, w) \in \AA$ is in $\AA_1$ or $\AA_2$. The set $\BB$ remains unchanged so for each $(i, j) \in E$ we have $(i, j) \in E_1$ or $(i, j) \in E_2$. Now if $r_1 \geq R(P_1)$ and $r_2 \geq R(P_2)$, then there is a minimal cover $C_1$ of position $P_1$ with $r(C_1) \leq r_1$ and the same for $P_2$. Now the set $C := C_1 \cup C_2$ is a cover of position $P$ with $r(C) \leq r_1 + r_2 < r$, which is a contradiction. Therefore we have $r_1 < R(P_1)$ or $r_2 < R(P_2)$ and D can maintain the condition by choosing such a position.

\smallskip

\noindent \textbf{\ddimove}: Let $k \leq d$ be the number chosen by S. The following position is $P' = (r-k, \AA', \BB')$. We first note that if $\MM_0 \in \mods(\AA) \cap \mods(\BB)$, then S must choose $k$ points from the version in $\AA$ and $n-k+1$ points from the version in $\BB$. Thus the next position $P'$ will have propositionally equivalent versions $(\MM_0, w) \in \AA'$ and $(\MM_0, v) \in \BB'$. By Lemma \ref{lem:samemodel} this gives D a winning strategy so we assume $\MM_0$ is only present on one side of the game. 

\smallskip

\noindent \textbf{Case $\AA$}: $\MM_0 \in \mods(\AA)$. For each $(\MM_0, w) \in \AA$, S chooses a set $P_{\MM_0,w}$ of $k$ points. Let $P_{\MM_0} := \bigcup_{w \in W} P_{\MM_0, w}$. 
We consider the following cases:

\smallskip

\noindent \textbf{1)} We have $\sum_{\pi_i \in \tp(P_{\MM_0})} n_i > k$. Let $(i, j) \in E$. Since the model $\MM_{i \rightarrow j}$ only differs from $\MM_0$ by $n'_i = n_i -1$ and $n'_j \geq n_j$, the model $\MM_{i \rightarrow j}$ has at least $k$ points with types from $\tp(P_{\MM_0})$. Thus, when S chooses the set $N_{\MM_{i \rightarrow j}}$ of $n-k+1$ points, it contains at least one point with a type from $\tp(P_{\MM_0})$. Thus the pair of propositionally equivalent models is maintained and $(i, j) \in E'$.

\smallskip
    
    \noindent \textbf{2)} We have $\sum_{\pi_i \in \tp(P_{\MM_0})} n_i = k$. Let $(i, j) \in E$. Assume $\pi_i \notin \tp(P_{\MM_0})$ or $\pi_j \in \tp(P_{\MM_0})$. Now the model $\MM_{i \rightarrow j}$ has at least $k$ points with types from $\tp(P_{\MM_0})$ and $(i, j) \in E'$ as in case 1. 

    Assume then that $\pi_i \in \tp(P_{\MM_0})$ and $\pi_j \notin \tp(P_{\MM_0})$. All edges of this kind being eliminated is an acceptable worst case, so we assume that $(i, j) \notin E'$.

\smallskip

Summing up Case $\AA$, the worst case is that S chooses a set $\tp(P_{\MM_0})$ of types with $\sum_{\pi_i \in \tp(P_{\MM_0})} n_i = k \leq d$ and eliminates all edges with $\pi_i \in \tp(P_{\MM_0})$ and $\pi_j \notin \tp(P_{\MM_0})$. 

\smallskip

\noindent \textbf{Case $\BB$}: $\MM_0 \in \mods(\BB)$. For each $(\MM_0, w) \in \BB$, S chooses a set $N_{\MM_0,w}$ of $n-k+1$ points. Let $N_{\MM_0} := \bigcup_{w \in W} N_{\MM_0, w}$ and let $\Pi$ be the complement of $\tp(N_{\MM_0})$. We note that $n_i < d$ for each $\pi_i \in \Pi$ since $\MM_0$ only has at most $k-1$ points with types from $\Pi$.

Let $(i,j) \in E$ and assume $\pi_i \in \Pi$ or $\pi_j \notin \Pi$. Now $\pi_i \notin \tp(N_{\MM_0})$ or $\pi_j \in \tp(N_{\MM_0})$ so the model $\MM_{i \rightarrow j}$ has at least $n-k+1$ points with types from $\tp(N_{\MM_0})$. Thus any set $P_{\MM_{i \rightarrow j},v}$ of $k$ points contains at least one point with a type from $\tp(N_{\MM_0})$. Thus the pair of propositionally equivalent models is maintained and $(i,j) \in E'$. 

Now assume $\pi_i \notin \Pi$ and $\pi_j \in \Pi$. If $\MM_{i \rightarrow j}$ has at least $n-k+1$ points with types from $\tp(N_{\MM_0})$, then $(i,j) \in E'$ as above. We thus assume that $\MM_{i \rightarrow j}$ has less than $n-k+1$ points with types from $\tp(N_{\MM_0})$, meaning it has at least $k$ points with types from $\Pi$. Since $n_j < d$ and $\MM_{i \rightarrow j}$ differs from $\MM_0$ only by $n'_i = n_i -1$ and $n'_j = n_j + 1$, the only remaining option is that $\MM_{i \rightarrow j}$ has exactly $k$ points with types from $\Pi$. Thus by the same reasoning $\MM_0$ has exactly $k-1$ points with types from $\Pi$ and since $k-1 < d$, we have $\sum_{i \in \Pi} n_i = k-1$. We again accept S eliminating these edges as a worst case and assume $(i,j) \notin E'$. 

\smallskip

Summing up Case $\BB$, the worst case is that S chooses a set $\Pi$ of types with $\sum_{\pi_i \in \Pi} n_i = k-1<d$ and eliminates all edges with $\pi_i \notin \Pi$ and $\pi_j \in \Pi$.

\smallskip

We now consider the condition $r-k < R(P')$ in the following position. From the above arguments we see that the only way for S to eliminate edges moving from $\G(\AA, \BB)$ to $\G(\AA', \BB')$ is to choose in each model $(\MM_0, w)$ or $(\MM_{i \rightarrow j}, v)$ in $\AA$ exactly all points that satisfy some set $\Pi$ of types. If $\MM_0 \in \mods(\AA)$, we have $\sum_{\pi_i \in \Pi} n_i = k \leq d$ and S can eliminate all edges from types in $\Pi$ to other types. If $\MM_0 \in \mods(\BB)$, we have $\sum_{\pi_i \in \Pi} n_i = k-1$ and $n_i < d$ for all $\pi_i \in \Pi$. In this case S can eliminate all edges from other types to types in $\Pi$. In both cases, the set $C = \{i \in \mathrm{supp}(\overline{n}) \mid \pi_i \in \Pi\}$ covers all eliminated edges. The cost of $C$ is $r(C) = \sum_{i \in C} n_i \leq k$. Let $C'$ be a cover of $P'$ with minimal cost so $r(C') = R(P')$. Now $C \cup C'$ is a cover of $P$ so $r < R(P) \leq R(P') + r(C) \leq R(P') + k$. Thus $r-k < R(P')$.

\noindent \textbf{\edimove}: Let $k < d$ be the number chosen by S. The following position is $P' = (r-k-1, \AA', \BB')$. As for the \ddimove, we may assume that $\MM_0$ is only present on one side of the game. 

\smallskip

\noindent \textbf{Case $\AA$}: $\MM_0 \in \mods(\AA)$. For each $(\MM_0, w) \in \AA$, S chooses a $k$ point set $P_{\MM_0, w}$ and $N_{\MM_0, w} = W \setminus P_{\MM_0, w}$. We denote $P_{\MM_0} := \bigcup_{w \in W} P_{\MM_0, w}$ and $N_{\MM_0} := \bigcup_{w \in W} N_{\MM_0, w}$. We consider the following cases:

\smallskip
 
 \noindent \textbf{1)} We have $\sum_{\pi_i \in \tp(P_{\MM_0})} n_i > k$. Thus there are propositionally equivalent $w \in P_{\MM_0}$ and $v \in N_{\MM_0}$. This means that in the following position $P'$ there are propositionally equivalent versions of the model $\MM_0$ on both sides of the game. D has a winning strategy from $P'$ by Lemma \ref{lem:samemodel}.

\smallskip

 \noindent \textbf{2)} We have $\sum_{\pi_i \in \tp(P_{\MM_0})} n_i = k$. Note that since $k < d$, also $n_i < d$ for all $\pi_i \in \tp(P_{\MM_0})$. Let $(i,j) \in E$.

    Assume $\pi_i, \pi_j \notin \tp(P_{\MM_0})$. Now the model $\MM_{i \rightarrow j}$ has exactly $k$ points with types from $\tp(P_{\MM_0})$. Thus if S chooses a $k+1$ point set $P_{\MM_{i \rightarrow j}}$, then one of those points $v'$ has a type $\pi_{\ell} \notin \tp(P_{\MM_0})$. By the definition of the models, $n_{\ell} \neq 0$ so there is a point $w'$ in the model $\MM_0$ of type $\pi_{\ell}$. Now there are propositionally equivalent $(\MM_{i \rightarrow j}, v') \in \AA'$ and $(\MM_0, w') \in \BB'$ so $(i, j) \in E'$. 

    Similarly, if S chooses for $\MM_{i \rightarrow j}$ a set $N_{\MM_{i \rightarrow j}}$ of $n-k+1$ points, then this set contains at least one point $v'$ of a type $\pi_{\ell} \in \tp(P_{\MM_0})$. Now there is a point $w' \in P_{\MM_0}$ of type $\pi_{\ell}$. Thus there are propositionally equivalent $(\MM_0, w') \in \AA'$ and $(\MM_{i \rightarrow j}, v') \in \BB'$ so $(i, j) \in E'$.

    If $\pi_i \in \tp(P_{\MM_0})$ or $\pi_j \in \tp(P_{\MM_0})$, we assume $(i, j) \notin E'$. As for the \ddimove, this is an acceptable worst case for our lower bound. 

    \smallskip

Summing up Case $\AA$, the worst case is that S chooses a set $\tp(P_{\MM_0})$ of types with $\sum_{\pi_i \in \tp(P_{\MM_0})} n_i = k < d$ and eliminates all edges $(i, j) \in E$ with $\pi_i \in \tp(P_{\MM_0})$ or $\pi_j \in \tp(P_{\MM_0})$.

\smallskip

\noindent \textbf{Case $\BB$}: $\MM_0 \in \mods(\BB)$. For each $(\MM_0, w) \in \BB$, S chooses either a set $P_{\MM_0, w}$ of $k+1$ points or a set $N_{\MM_0, w}$ of $n-k+1$ points. There are three cases:

\smallskip

\noindent \textbf{1)} Assume first that there are $(\MM_0, w), (\MM_0, w') \in \BB$ with sets $P_{\MM_0, w}$ and $N_{\MM_0, w'}$ chosen. Since $k+1 + n-k+1 > n$, there is $v \in P_{\MM_0, w} \cap N_{\MM_0, w'}$. In the following position $P'$ the model $(\MM_0, v)$ is on both sides of the game so by Lemma \ref{lem:samemodel} D has a winning strategy from $P'$.

\smallskip

 \noindent \textbf{2)} Assume then that S chooses a set $P_{\MM_0, w}$ of $k+1$ points for each $(\MM_0, w) \in \BB$ and let $P_{\MM_0} = \bigcup_{w \in W} P_{\MM_0, w}$. Let $(i, j) \in E$ and let $(\MM_{i \rightarrow j}, v) \in \AA$ and $(\MM_0, w) \in \BB$ be the corresponding models. 

    Assume $\pi_i \notin \tp(P_{\MM_0})$ or $\pi_j \in \tp(P_{\MM_0})$. Now the model $\MM_{i \rightarrow j}$ has at least $k+1$ points with types from $\tp(P_{\MM_0})$. Thus the set $N_{\MM_{i \rightarrow j}, v}$ of size $n-k$ has at least one point with a type from $\tp(P_{\MM_0})$. Thus the propositionally equivalent pair of models is maintained and $(i, j) \in E'$.

    Assume $\pi_i \in \tp(P_{\MM_0})$ and $\pi_j \notin \tp(P_{\MM_0})$. If $\MM_{i \rightarrow j}$ has at least $k+1$ points with types from $\tp(P_{\MM_0})$, the above argument again works and $(i, j) \in E'$. Since $k < d$ and $\MM_{i \rightarrow j}$ only differs from $\MM_0$ by $n'_i = n_i -1$ and $n'_j \leq n_j +1$, the only remaining option is that $\MM_{i \rightarrow j}$ has exactly $k$ points with types from $\tp(P_{\MM_0})$. We obtain $\sum_{i \in \tp(P_{\MM_0})} n'_i = k$. We again assume as a worst case that $(i,j) \notin E'$. 

    \smallskip

  \noindent \textbf{3)} Finally assume that S chooses a set $N_{\MM_0, w}$ of $n-k+1$ points for each $(\MM_0, w) \in \BB$ and let $N_{\MM_0} = \bigcup_{w \in W} N_{\MM_0, w}$. Let $\Pi$ be the complement of $\tp(N_{\MM_0})$.

    Assume $\pi_i \in \Pi$ or $\pi_j \notin \Pi$. Now $\pi_i \notin \tp(N_{\MM_0})$ or $\pi_j \in \tp(N_{\MM_0})$ so the model $\MM_{i \rightarrow j}$ has at least $n-k+1$ points with types from $\tp(N_{\MM_0})$. Thus the set $P_{\MM_{i \rightarrow j}, v}$ of size $k$ has at least one point with a type from $\tp(N_{\MM_0})$. Thus the propositionally equivalent pair of models is maintained and $(i,j) \in E'$.

    Assume $\pi_i \notin \Pi$ and $\pi_j \in \Pi$. If $\MM_{i \rightarrow j}$ has at least $n-k+1$ points with types from $\tp(N_{\MM_0})$, the above argument again works and $(i,j) \in E'$. 
    
    We assume $\MM_{i \rightarrow j}$ has less than $n-k+1$ points with types from $\tp(N_{\MM_0})$, meaning it has at least $k$ points with types from $\Pi$. Since $k < d$ and $\MM_{i \rightarrow j}$ only differs from $\MM_0$ by $n'_i = n_i -1$ and $n'_j \leq n_j +1$, the only remaining option is that $\MM_{i \rightarrow j}$ has exactly $k$ points with types from $\Pi$. We obtain $\sum_{i \in \Pi} n'_i = k$. We again assume as a worst case that $(i,j) \notin E'$. Edges eliminated this way are ones from other types to types in $\Pi$. In particular, we note that for any edge $(i,j)$ eliminated this way, $n_j < d$ since $\MM_0$ has at most $n-(n-k+1) = k-1$ points of type $\pi_j$.

    \smallskip

Summing up Case $\BB$, the worst case is that S chooses a set $\Pi$ of types  and eliminates either all edges $(i, j) \in E$ with $\pi_i \in \Pi$ or $\pi_j \notin \Pi$ or all edges $(i, j) \in E$ with $\pi_i \notin \Pi$ or $\pi_j \in \Pi$. In both cases $\sum_{i \in \Pi} n'_i = k$ for the model $\MM_{i \rightarrow j}$ and for the latter case $n_j < d$.

\smallskip

We now consider the condition $r-k-1 < R(P')$ in the following position. From the above arguments, we see that the only way for S to eliminate edges moving from $\G(\AA, \BB)$ to $\G(\AA', \BB')$ is to choose in each model $(\MM_0, w)$ or $(\MM_{i \rightarrow j}, v)$ in $\AA$ exactly all points that satisfy some set $\Pi$ of types. If $\MM_0 \in \mods(\AA)$, we have $\sum_{\pi_i \in \Pi} n_i = k < d$ and S can eliminate all edges to and from the set $\Pi$ of types. If $\MM_0 \in \mods(\BB)$, we have $\sum_{\pi_i \in \Pi} n'_i = k < d$ and S can eliminate either all edges to or all edges from the set $\Pi$ of types. In particular, if $n_j = d$ for $\pi_j \in \Pi$, then only outgoing edges can be eliminated. In all cases, the set $C = \{i \in \mathrm{supp}(\overline{n}) \mid \pi_i \in \Pi\}$ covers all eliminated edges. The cost of $C$ is $r(C) = \sum_{i \in C} n_i \leq k+1$ since $\sum_{i \in C} n_i \leq \sum_{i \in C} n'_i+1$. Let $C'$ be a cover of $P'$ with minimal cost so $r(C') = R(P')$. Now $C \cup C'$ is a cover of $P$ so $r < R(P) \leq R(P') + r(C) \leq R(P') + k + 1$. Thus $r-k-1 < R(P')$.
\end{proof}


A lower bound for the description complexity of an arbitrary class $M_{\overline{n}}$ can now be obtained by simply calculating the minimum cost of a cover of $\G(\AA_{\overline{n}},\BB_{\overline{n}})$. 

\begin{lemma}\label{thm:two-different-indexes-dcomplexity}
Let $M_{\overline{n}}$ be a class with $n_i = d$ for at least two different $i \in I$.
Then $C(M_{\overline{n}}) \geq \sum\limits_{i \in I} n_i$.
\end{lemma}
\begin{proof}
We assume that $n_1 = d$. We consider the graph $\G(\AA_{\overline{n}}, \BB_{\overline{n}})= \{(i,j) \in \mathrm{supp}(\overline{n}) \mid i \neq 1 \text{ or } n_j =d\}$. This graph has edges from all types $\pi_i$ with $n_i \neq 0$ to each other, with the exception that there are no edges from $\pi_1$ to $\pi_j$ with $n_j < d$. 

We first note that $C = \mathrm{supp}(\overline{n})$ is a cover of $\G(\AA_{\overline{n}}, \BB_{\overline{n}})$ with cost $\sum_{i \in I} n_i$. For $C' \neq C$ if $i \notin C'$ with $i \neq 1$, then by the definition of a cover, the edge $(i,1) \in E$ is not covered, since $i \notin C$ and $n_1 = d$. If $1 \notin C'$, then let $j \in I$ be the other index with $n_j = d$ besides 1. Now the edge $(1, d) \in E$ is not covered since $1 \notin C$ and $n_j = d$. We see that $C$ is a minimal cost cover and by Lemma \ref{coverlemma} and Theorem \ref{thm:game} the claim holds.
\end{proof}

\begin{lemma}\label{thm:one-index-dcomplexity}
Let $M_{\overline{n}}$ be a class with $n_i = d$ for at most one $i \in I$. Let $\pi_j$ be one of the types with the most realizing points in $M_{\overline{n}}$. 
Then $C(M_{\overline{n}}) \geq  \sum_{i \in I \setminus \{j\}} n_i = n - |\pi_j|$.
\end{lemma}
\begin{proof}
We assume $n_1 = \max\{n_i \mid i \in I\}$. Thus $\pi_1$ is one of the types with the most realizing points and if $n_1 < d$, then $n_i < d$ for all $i \in I$. 

We consider the graph $\G(\AA_{\overline{n}}, \BB_{\overline{n}})$. From the definition we get $G(\AA_{\overline{n}}, \BB_{\overline{n}})= \{(i,j) \in \mathrm{supp}(\overline{n}) \mid i \neq 1\}$. This graph has edges from all types $\pi_i$ with $n_i \neq 0$ to each other, with the exception that there are no edges originating from $1$. 

We first note that $C = \mathrm{supp}(\overline{n}) \setminus \{1\}$ is a cover of $\G(\AA_{\overline{n}}, \BB_{\overline{n}})$ with $r(C) = \sum_{i \in I \setminus \{1\}} n_i = n - |\pi_1|$. Clearly $\mathrm{supp}(\overline{n})$ is a cover with higher cost. 

Let $C' \subset \mathrm{supp}(\overline{n})$. If there are $i, j \in \mathrm{supp}(\overline{n})$ with $i, j \notin C'$, then the edge $(i, j)$ or $(j,i)$ is in $E$ and is not covered. If there is only one $i \in \mathrm{supp}(\overline{n})$ with $i \notin C'$ and $i \neq 1$, then $r(C') = \sum_{i \in I \setminus \{j\}} n_i \geq r(C)$ since $\pi_1$ is one of the types with the largest $n_i$. We see that $C$ is a minimal cost cover and by Lemma \ref{coverlemma} and Theorem \ref{thm:game} the claim holds.
\end{proof}

We sum up the above lemmas into the Theorem below:
\begin{theorem}\label{thm:main-descriptive-complexity}
Let $\overline{n}$ be an $(n,d)$-admissible tuple and let $M_{\overline{n}}$ be the corresponding equivalence class of $\equiv_{d}$. If $n_i = d$ for at least two $i \in I$, then $C(M_{\overline{n}}) \geq \sum_{i \in I} n_i$. Otherwise $C(M_{\overline{n}}) \geq  \sum_{i \in I \setminus \{j\}} n_i = n - |\pi_j|$.
\end{theorem}

\section{A monotone connection}\label{monotoneconnection}

Consider the following strict partial orders on the set of all $(n,d)$-admissible tuples.
\begin{enumerate}
    \item $\overline{n} <_s \overline{n}'$ iff $|M_{\overline{n}}| < |M_{\overline{n}'}|$
    \item $\overline{n} <_c \overline{n}'$ iff $C(M_{\overline{n}}) < C(M_{\overline{n}'})$
\end{enumerate}
Note that neither $|M_{\overline{n}}|=|M_{\overline{n}'}|$ nor $C(M_{\overline{n}})=C(M_{\overline{n}'})$ necessarily entails that $\overline{n} = \overline{n}'$, as demonstrated by the tuples $(1,2,d)$ and $(2,1,d)$. We let $\leq_s$ and $\leq_c$ denote the partial orders obtained by adding loops to $<_s$ and $<_c$ respectively. The main purpose of this section is to show that there exists a natural and non-trivial partial order which is contained both in $\leq_s$ and in $\leq_c$ (provided that $n$ is sufficiently large w.r.t. $d$), which then gives us a monotone connection between sizes of model classes and their description complexities. To establish this result, we will use the heavy machinery developed in the previous section together with further combinatorial arguments that involve $r$-associated Stirling numbers.

\subsection{$r$-associated Stirling numbers}

We will use $r$-associated Stirling numbers to count the number of models in a given model class. Given positive integers $n,m,r$ such that $n \geq mr$, we define
\[\stirling{n}{m}_{\geq r}\]
to be the number of partitions of $[n]$ which partition $[n]$ into $m$, each set having size at least $r$. When $r = 1$ these numbers are also known as the Stirling numbers of the second kind and they simply count the number of partitions of $[n]$ into $m$ sets \cite{concretemathematics}.

The following lemma will play the key role when we estimate the sizes of the model classes. Even though the proof is simple and elementary, we were not able to find these estimates in the existing literature.

\begin{lemma}\label{lemma1}
    Let $n,m,r \in \mathbb{Z}_+$. Suppose that $n \geq mr$. Then
    \[\frac{m^n}{m^{mr}} \leq \stirling{n}{m}_{\geq r} \leq \frac{m^n}{m!}\]
\end{lemma}
\begin{proof}
    For the upper bound note that
    \[\stirling{n}{m}_{\geq r} \leq \stirling{n}{m}_{\geq 1} \leq \frac{m^n}{m!}.\]
    The last inequality follows from the fact that $m^n$ counts the number of mappings $f:[n] \to [m]$, while $m! \cdot \stirling{n}{m}_{\geq 1}$ only counts those $f$ that are surjections. For the lower bound, we use the fact that $m! \cdot \stirling{n}{m}_{\geq r}$ counts the number of mappings $f:[n] \to [m]$ which have the property that for every $1 \leq k \leq m$ we have $|f^{-1}(\{k\})| \geq r$. To give a lower bound on the number of such functions, we count the number of mappings $f:[n] \to [m]$ which have the property that the sets
    $\{(k - 1) \cdot r + 1, \dots, k \cdot r\},$
    where $1 \leq k \leq m$, are mapped to distinct elements. Since the remaining $n - mr$ elements can be mapped arbitrarily, the number of such mappings is simply $m! \cdot m^{n - mr}$. Thus \[m! \cdot m^{n-mr} \leq m! \cdot \stirling{n}{m}_{\geq r},\]
    giving the wanted lower bound after dividing by $m!$.
\end{proof}

We note that if $m$ and $r$ are much smaller than $n$ --- as they will be in our applications --- the estimates of Lemma \ref{lemma1} are (perhaps surprisingly) quite sharp.

One important consequence of Lemma \ref{lemma1} is the following lemma. Again, we emphasize that we were unable to find an estimate of this form in the existing literature.

\begin{lemma}\label{lemma2}
    Let $n,m,r \in \mathbb{Z}_+$. If $n \geq mr + 1$, then
    \[\stirling{n}{m}_{\geq r} \leq \frac{m^{mr+1}}{m!} \stirling{n-1}{m}_{\geq r}.\]
\end{lemma}
\begin{proof}
    Follows immediately from Lemma \ref{lemma1}.
\end{proof}

If $m$ and $r$ are fixed, Lemma \ref{lemma2} bounds the growth rate of $r$-associated Stirling numbers as $n$ increases.

\subsection{Combinatorics of model classes}

In this subsection we use the above results on $r$-associated Stirling numbers to investigate the sizes of model classes in terms of their $(n, d)$-admissible tuples. 
We begin with the following lemma gives a simple and closed formula for the size of a model class.

\begin{lemma}\label{lemma:counting-models}
    Let $\overline{n}$ be an $(n,d)$-admissible tuple. Let $i_1,\dots,i_k \in I$ be the indices for which $n_{i_\ell} < d$. We then have that
    \[|M_{\overline{n}}| = \binom{n}{n_{i_1},\dots,n_{i_k},m} \cdot k_d! \cdot \stirling{m}{k_d}_{\geq d},\]
    where $m := n - \sum_{i \not\in \{i_1,\dots,i_k\}} n_i$ and $k_d := t - k$.
\end{lemma}
\begin{proof}
    Each model in $M_{\overline{n}}$ can be constructed as follows.

    \smallskip
    
    \noindent\textbf{(1)} We first pick $k$ subsets of $\{1,\dots,n\}$ of sizes $n_{i_1}, \dots, n_{i_k}$ and define that each element in the $i_\ell$th set realizes the $i_\ell$th type. The number of ways this can be done is given by $\binom{n}{n_{i_1},\dots,n_{i_k},m}$.

    \smallskip
    
    \noindent \textbf{(2)} We then partition the remaining subset of size $m$ to $k_d$ pieces, each piece having size at least $d$. The number of ways this can be done is given by the associated Stirling number $\stirling{m}{k_d}_{\geq d}$.

    \smallskip
    
    \noindent \textbf{(3)} For each piece we select a unique type from the remaining types --- the number of which is $k_d$ --- and define that each element in a piece realizes the type associated with that piece. The number of ways this can be done is given by $k_d!$.

    \smallskip
    
    \noindent Multiplying the above factors gives us the result.
\end{proof}

The following lemmas establish how the size of a model class changes when we modify its admissible tuple.

\begin{lemma}\label{lemma:one-more-element}
    Fix $d \in \mathbb{Z}_+$ and let $\overline{n}$ be an $(n,d)$-admissible tuple. Suppose that $n$ is sufficiently large with respect to $d$ and $|\tau|$.
    Then for every $i \in I$ such that $n_i < d - 1$ we have that
    \[|M_{\overline{n}}| < |M_{\overline{n}'}|,\]
    where $n_i' = n_i + 1$ and $n_j' = n_j$, for every $j \neq i$.
\end{lemma}
\begin{proof}
    For notational simplicity we assume that $n_\ell < d$ iff $\ell \leq k$. By Lemma \ref{lemma:counting-models} the inequality that we need to establish is
    \begin{align*}
        &\binom{n}{n_1,\dots,n_k,m} \cdot k_d! \cdot \stirling{m}{k_d}_{\geq d} \\
        &< \binom{n}{n_1,\dots,n_i + 1,\dots,n_k,m - 1} \cdot k_d! \cdot \stirling{m-1}{k_d}_{\geq d},
    \end{align*}
    where $m = n - \sum_{\ell = 1}^k n_\ell$ and $k_d = t - k$. Note that since $n$ is large enough w.r.t. $d$, $k_d \geq 1$. Simplifying this gives us the equivalent inequality
    \[\stirling{m}{k_d}_{\geq d} < \frac{m}{n_i + 1} \stirling{m-1}{k_d}_{\geq d},\]
    which follows from Lemma \ref{lemma2}, as long as
    \[\frac{m}{n_i + 1} > \frac{k_d^{k_d d + 1}}{k_d!} \Leftrightarrow n > (n_i + 1)\frac{k_d^{k_d d + 1}}{k_d!} + \sum_{\ell = 1}^k n_\ell\]
    Using $n_\ell < d$, which holds for every $1 \leq \ell \leq k$, and $1 \leq k_d \leq t$ gives us the desired result.
\end{proof}

\begin{lemma}\label{lemma:one-more-d}
    Fix $d \in \mathbb{Z}_+$ and let $\overline{n}$ be an $(n,d)$-admissible tuple. Suppose that $n$ is sufficiently large with respect to $d$ and $|\tau|$. Then for every $i \in I$ such that $n_i < d$ we have that
    \[|M_{\overline{n}}| < |M_{\overline{n}'}|,\]
    where $n_j' = d$, when $j = i$, and $n_j' = n_j$ otherwise.
\end{lemma}
\begin{proof}
    For notational simplicity we assume that $n_\ell < d$ iff $\ell \leq k$. By Lemma \ref{lemma:counting-models} the inequality that we need to establish is
    \begin{align*}
        &\binom{n}{n_1,\dots,n_i,\dots,n_k,m} \cdot k_d! \cdot \stirling{m}{k_d}_{\geq d} \\
        &< \binom{n}{n_1,\dots,n_{i-1},n_{i+1},\dots,n_k,m + n_i} \\
        & \cdot (k_d + 1)! \cdot \stirling{m + n_i}{k_d + 1}_{\geq d},
    \end{align*}
    where $m = n - \sum_{\ell = 1}^k n_\ell$ and $k_d = t - k$. Note that since $n$ is large enough w.r.t. $d$, $k_d \geq 1$. Simplifying this gives us the equivalent inequality
    \[\stirling{m}{k_d}_{\geq d} < \frac{m!n_i!}{(m + n_i)!} \cdot (k_d + 1) \cdot \stirling{m + n_i}{k_d + 1}_{\geq d}\]
    It follows from Lemma \ref{lemma1} that
    \[\stirling{m + n_i}{k_d + 1}_{\geq d} \geq \frac{(k_d + 1)^{m + n_i}}{(k_d + 1)^{(k_d + 1)d}}\]
    and
    \[\stirling{m}{k_d}_{\geq d} \leq \frac{k_d^m}{k_d!}\]
    Hence we only have to show that
    \[\frac{k_d^m}{k_d!} < \frac{m!n_i!}{(m + n_i)!} \cdot (k_d + 1) \cdot \frac{(k_d + 1)^{m + n_i}}{(k_d + 1)^{(k_d + 1)d}}\]
    or equivalently that
    \begin{align*}
    \frac{(k_d + 1)^{(k_d + 1)d}}{(k_d + 1)^{n_i + 1}k_d!n_i!} 
    &< \frac{m!(k_d + 1)^{m}}{k_d^m(m + n_i)!} \\
    &= \frac{(k_d + 1)^{m}}{k_d^m\prod_{j = 0}^{n_i - 1}(m + n_i - j)} \\
    &= \bigg(\underbrace{\frac{k_d + 1}{k_d}}_{> 1}\bigg)^{m} \bigg/ \prod_{j = 0}^{n_i - 1}(m + n_i - j)
    \end{align*}
    Recall $0 \leq n_i < d$ and $1 \leq k_d \leq t$. Hence in the above inequality the left hand side is a constant, and thus it suffices to show that the right hand side formula tends to infinity as $n$ grows (recall $m = n - \sum_{\ell = 1}^k n_\ell$, i.e., $m$ is just $n$ minus a constant). However, this is clear, since the right hand side is of the form $f(n)/g(n)$, where $f$ grows exponentially w.r.t. $n$ while $g$ grows only polynomially w.r.t. $n$.
\end{proof}

\subsection{Connecting size and description complexity}

Let $P_{n,d}$ denote the set of all $(n,d)$-admissible tuples. We define a natural partial order $\preceq$ on $P_{n,d}$ as follows: $\overline{n} \preceq \overline{n}'$ if and only if $n_i \leq n_i'$, for every $i \in I$. By writing $\overline{n} \prec \overline{n}'$ we mean that $\overline{n} \preceq \overline{n}'$ and $\overline{n} \neq \overline{n}'$.

\begin{lemma}\label{lemma:class-size-monotonicity}
    Suppose that $n$ is sufficiently large with respect to $d$ and $|\tau|$. Let $\overline{n}$ and $\overline{n}'$ be $(n,d)$-admissible tuples such that $\overline{n} \prec \overline{n}'$. Then
    \[|M_{\overline{n}}| < |M_{\overline{n}'}|\]
\end{lemma}
\begin{proof}
    It suffices to show that if $\overline{n}'$ is an immediate successor of $\overline{n}$, then $|M_{\overline{n}}| < |M_{\overline{n}'}|$. Now, there must exist exactly one $i \in I$ such that $n_i' = n_i + 1$ and $n_j' = n_j$ for every $j \neq i$. If $n_i' < d$, then the claim follows from Lemma \ref{lemma:one-more-element}, while if $n_i' = d$, then the claim follows from Lemma \ref{lemma:one-more-d}.
\end{proof}

Recall the constant $c_\tau = 2^{|\tau|}(2|\tau| + 1) - 1$ from the previous section. We define yet another partial order $\preceq_{\tau}$ on $P_{n,d}$ as follows: $\overline{n} \preceq_\tau \overline{n}'$ if and only if the following two conditions hold:
\begin{enumerate}
    \item For every $i \in I$ we have that $n_i \leq n_i'$.
    \item Either $\overline{n} = \overline{n}'$ or $\sum_{i \in I} (n_i' - n_i) > c_\tau$.
\end{enumerate}
Roughly speaking $\overline{n} \preceq_\tau \overline{n}'$ means that if the tuples are distinct, then the distance between them w.r.t. to the order $\preceq$ is more than $c_\tau$. Again, by writing $\overline{n} \prec_\tau \overline{n}'$ we mean that $\overline{n} \preceq_\tau \overline{n}'$ and $\overline{n} \neq \overline{n}'$. For example $(1,d) \prec_\tau (2 + c_\tau,d)$, but $(1,d) \not\preceq_\tau (1 + c_\tau,d)$.

\begin{lemma}\label{lemma:description-complexity-monotonicity}
    Suppose that $n$ is sufficiently large with respect to $d$ and $|\tau|$. Let $\overline{n}$ and $\overline{n}'$ be $(n,d)$-admissible tuples such that $\overline{n} \prec_\tau \overline{n}'$. Then we have that
    \[C(M_{\overline{n}}) < C(M_{\overline{n}'})\]
\end{lemma}
\begin{proof}
    Suppose first that $d$ occurs at least twice in $\overline{n}'$. Theorem \ref{thm:main-descriptive-complexity} entails that $C(M_{\overline{n}'}) \geq \sum_{i\in I} n_i'$. We now have two cases based on whether or not $d$ occurs at least twice in $\overline{n}$. Suppose first that $d$ occurs exactly once in $\overline{n}$ say, $n_j = d$. Then $M_{\overline{n}}$ can be defined by a formula of size 
    \begin{align*}
        c_\tau + \sum_{i \in I - \{j\}} n_i
        < \sum_{i \in I - \{j\}} n_i' \leq C(M_{\overline{n}'})
    \end{align*}
    and hence $C(M_{\overline{n}}) < C(M_{\overline{n}'})$. On the other hand, if $d$ occurs at least twice in $\overline{n}$, then $M_{\overline{n}}$ can be defined by a formula of size
    \begin{align*}
        c_\tau + \sum_{i \in I} n_i 
        < \sum_{i \in I} n_i' \leq C(M_{\overline{n}'})
    \end{align*}
    and hence $C(M_{\overline{n}}) < C(M_{\overline{n}'})$.

    Suppose then that $d$ occurs exactly once in $\overline{n}'$, say $n_j' = d$. Theorem \ref{thm:main-descriptive-complexity} entails that $C(M_{\overline{n}'}) \geq \sum_{i \in I - \{j\}} n_i'$. Since $\overline{n} \prec_\tau \overline{n}'$, we have that $n_i < d$, for every $i \neq j$. Since $d$ must occur at least once in $\overline{n}$ --- as $\overline{n}$ is $(n,d)$-admissible --- we have that $n_j = d$. Now $M_{\overline{n}}$ can be defined by a formula of size 
    \begin{align*}
        c_\tau + \sum_{i \in I - \{j\}} n_i
        < \sum_{i \in I - \{j\}} n_i'
        \leq C(M_{\overline{n}'})
    \end{align*}
    and hence $C(M_{\overline{n}}) < C(M_{\overline{n}'})$.
\end{proof}

Recall the partial orderings $\leq_s$ and $\leq_c$ on $P_{n,d}$, which were introduced at the beginning of this section. The following theorem formalizes a connection between sizes of model classes and their description complexities.

\begin{theorem}\label{thm:connecting-size-and-complexity}
    If $n$ is sufficiently large with respect to $d$, then
    \[\preceq_\tau \ \subseteq \ \leq_s \cap \leq_c.\]
    In particular, if $\overline{n}$ and $\overline{n}'$ are two distinct $\preceq_\tau$-comparable tuples, then
    \[|M_{\overline{n}}| < |M_{\overline{n}'}| \Leftrightarrow C(M_{\overline{n}}) < C(M_{\overline{n}'}).\]

    
\end{theorem}
\begin{proof}
    Suppose that $\overline{n} \prec_\tau \overline{n}'$. Lemmas \ref{lemma:class-size-monotonicity} and \ref{lemma:description-complexity-monotonicity} guarantee that if $n$ is large enough w.r.t. $d$, then $|M_{\overline{n}}| < |M_{\overline{n}'}|$ and $C(M_{\overline{n}}) < C(M_{\overline{n}'})$. Hence $\overline{n} \leq_s \overline{n}'$ and $\overline{n} \leq_c \overline{n}'$, which proves the first claim. The second claim follows directly from the first.
\end{proof}

The intuitive content of the above theorem is that the partial order $\preceq_\tau$ approximates both $\leq_s$ and $\leq_c$. Hence $\preceq_\tau$ can be viewed as a highly non-trivial monotone connection between $\leq_s$ and $\leq_c$.

We note that the above theorem also works for Boltzmann entropy. This is because the order $\leq_s$ is the same as the order based on Boltzmann entropy, since logarithm is an increasing function. We formulate the latter of the two claims as a corollary.

\begin{corollary}
   Assume $n$ is sufficiently large with respect to $d$. If $\overline{n}$ and $\overline{n}'$ are two distinct $\preceq_\tau$-comparable tuples, then
    \[
    H_B(M_{\overline{n}}) < H_B(M_{\overline{n}'}) \Leftrightarrow C(M_{\overline{n}}) < C(M_{\overline{n}'}).
    \]
\end{corollary}

\section{The phase transitions of class size distributions}\label{phase}

In this section we move our attention from single classes to the entire probability distribution given by $\equiv_d$. By allowing $d$ to depend on $n$, we obtain qualitative results which link the growth rate of $d$ with the emergence of a dominating class, i.e., a class which contains almost all the models. For fixed $n$, we also obtain quantitative results on how the relationship between $d$ and $n$ determines whether there exists a class in $\equiv_d$ which contains majority of all the models of size $n$. We will also point out consequences of these results on the ``average-case'' expressive power of $\gmlu_d$. 

Throughout this section we continue to use our previous convention that $t = 2^{|\tau|}$ denotes the number of types $\pi$ of the alphabet $\tau$. 
We start with the following observation, which gives a sense of what happens in the distribution, when the counting depth is increased.
\begin{proposition}\label{prop:shannon-boltzmann-d}
    Suppose that $d < d' \leq n/2$. Then
    \[H_S(\equiv_d) < H_S(\equiv_{d'}) \text{ and } H_B(\equiv_{d}) >H_B(\equiv_{d'}).\]
%
%
%
    Furthermore, for every $n/2 \leq d \leq d'$ we have that
    \[H_S(\equiv_d) = H_S(\equiv_{d'}) \text{ and } H_B(\equiv_{d}) = H_B(\equiv_{d'}).\]
    %
%
%
%
%
%
%
%
\end{proposition}
\begin{proof}
    By Proposition \ref{prop:shannon_boltzmann} $H_S(\equiv_{d}) + H_B(\equiv_{d}) = |\tau|n$. Hence it suffices to establish the claims in the case of Boltzmann entropy. We first note that if $n/2 \leq d \leq d'$, then $H_B(\equiv_{d}) = H_B(\equiv_{d'})$, since the logic $\gmlu_d$ can already specify each structure up to isomorphism. 
    
    Suppose then that $d < d' \leq n/2$. Consider the class $M_{\overline{n}}$, where $\overline{n} = (0, \dots, 0, d, d)$. As the depth is increased to $d' > d$, this class is divided to at least two smaller classes with tuples $(0, \dots, 0, d, d')$ and $(0, \dots, 0, d', d)$. Meanwhile, clearly no class increases in size so we see that $H_B(\equiv_{d}) > H_B(\equiv_{d'})$.
\end{proof}

We take this opportunity to point out an easy corollary of the previous result. In \cite{stacsarxiv} it was established that $H_B(\equiv_n) \sim |\tau|n$, by which we mean that $\lim_{n \to \infty} H_B(\equiv_n)/|\tau|n = 1$. When combined with Proposition \ref{prop:shannon-boltzmann-d}, this result yields quite directly the following.

\begin{corollary}
    For any counting depth $d(n)$, we have 
    \[H_B(\equiv_{d(n)}) \sim |\tau|n\]
\end{corollary}
\begin{proof}
    W.l.o.g. we assume that $d(n) \leq n$, for every $n$. Since $\log(|M|) \leq |\tau|n$, for any class $M$, we have that $H_B(\equiv_{d(n)}) \leq |\tau|n$. Since $H_B(\equiv_n) \sim |\tau|n$, for every $\varepsilon > 0$ we have that if $n$ is large enough, then $H_B(\equiv_n) \geq (1 - \varepsilon)|\tau|n$. Since $H_B(\equiv_{d(n)}) \geq H_B(\equiv_n)$, for every $\varepsilon > 0$ we have that $H_B(\equiv_{d(n)}) \geq (1 - \varepsilon)|\tau|n$, provided that $n$ is sufficiently large. Combining these bounds yields the desired result.
\end{proof}

Thus, from an asymptotic point of view, the dependence of the counting depth $d$ on $n$ has no effect on the Boltzmann entropy of $\equiv_d$. By virtue of Proposition \ref{prop:shannon_boltzmann}, the same is true for the Shannon entropy of $\equiv_d$.
We proceed now to further analyze the effect of counting depth on the distribution. To formulate some of our results, we will use standard asymptotic notation, which we recall here. Let $f,g:\mathbb{N} \to \mathbb{R}_{> 0}$. We use $f = o(g)$ to denote that $\lim_{n \to \infty} f(n) / g(n) = 0$. Furthermore, we use $f = \omega(g)$ to denote that $\lim_{n \to \infty} f(n)/g(n) = \infty$.

We first show that if the counting depth grows slowly enough with respect to $n$, then the distribution contains a class which contains almost all of the models of size $n$. 
More formally, fix a function $d:\mathbb{Z}_+\rightarrow\mathbb{N}$ and thereby a sequence $\equiv_{d(n)}$ of equivalence relations. A \textbf{class sequence} $M(n)$ with respect to the sequence $\equiv_{d(n)}$ is a function that outputs a single class for each individual equivalence relation in the sequence. Each class sequence $M(n)$ is naturally associated with the corresponding \textbf{probability sequence} 
\[p_n := p_{\equiv_{d(n)}}(M(n)) = \frac{|M(n)|}{2^{n|\tau|}}.\]
We say that $\equiv_{d(n)}$ has a \textbf{dominating class} if there exists a class sequence $M(n)$ such that $p_n\rightarrow 1$ as 
$n\rightarrow \infty$.
Then the class sequence $M(n)$ is said to \textbf{dominate} $\equiv_{d(n)}$. 
Intuitively, a class (sequence) is dominating if a random model of size $n$ belongs to it with limit probability one. 

Our proof uses the well-known Chernoff bounds. The following lemma will require the lower-tail estimate, while the upper-tail estimate will be used later in this section.

\begin{proposition}[Chernoff bounds \cite{probandcomputing}]
    Let $X := \sum_{i = 1}^n X_i$ be a sum of independent $0$-$1$-valued random variables, where $X_i = 1$ with probability $p$ and $X_i = 0$ with probability $1 - p$. Then for every $\delta \geq 0$ we have that
    \begin{align*}
        & \text{\textbf{(Lower tail)} } \Pr[X \leq (1 - \delta)np] \leq e^{-\delta^2\frac{np}{2}}\\
        & \text{\textbf{(Upper tail)} } \Pr[X \geq (1 + \delta)np] \leq e^{-\delta^2\frac{np}{2 + \delta}}
    \end{align*}
\end{proposition}

\begin{lemma}\label{lemma:dominating-class}
    For any $f(n) = \omega(\sqrt{n})$, if the counting depth is $d(n) \leq n/t - f(n)$, then the class sequence $M_{\overline{n}}$, where $\overline{n} = (d(n), \dots, d(n))$, dominates $\equiv_{d(n)}$.
\end{lemma}
\begin{proof}
    We will show that for each $f(n)$, such that $f(n) = \omega(\sqrt{n})$ and $f(n) \leq n/t$, we have that with limit probability one a random model realizes each type more than $n/t - f(n)$ times. Given a type $\pi$, we let $X_\pi$ denote a random variable which counts the number of times $\pi$ is realized. Now $X_\pi := \sum_{1 \leq i \leq n} X_{\pi,i}$, where $X_{\pi,i}$ is an indicator random variable for the event that the element $i$ realizes the type $\pi$. Since the success probability of $X_{\pi,i}$ is $t^{-1}$, we have that $\mathrm{E}(X_\pi) = n/t$.
    
    Now, Chernoff bound give us that for every $n$ and for every $0 \leq \delta$ we have
    \[\Pr\bigg[X_\pi \leq (1 - \delta) \frac{n}{t}\bigg] \leq e^{-\delta^2 \frac{n}{2t}}.\]
    Note that $\delta$ can indeed depend on $n$. Setting $\delta(n) := \sqrt{2tg(n)/n}$, for any $g(n) = \omega(1)$, we obtain that
    \[e^{-\delta^2 \frac{n}{2t}} = e^{-g(n)} \to 0.\]
    Furthermore
    \[\delta(n) \frac{n}{t} = \underbrace{\sqrt{2t} \cdot \frac{1}{t}}_{=: C} \cdot \sqrt{g(n)} \cdot \sqrt{n} = C \sqrt{ng(n)}.\]
    Hence for any function $g(n) = \omega(1)$ we have that with high probability the type $\pi$ is realized more than $n/t - C \sqrt{n g(n)}$ many times. Since $\pi$ was arbitrary, it follows from the union bound that with high probability every type $\pi$ is realized more than $n/t - C \sqrt{n g(n)}$ times. Now, if $f(n) = \omega(\sqrt{n})$, then by setting $g(n) = (f(n)/C \sqrt{n})^2$ we obtain that every type is realized more than $n/t - f(n)$ times.
\end{proof}

We showed that when the counting depth is low enough, the distribution has a dominating class. Next we will show that for a high enough counting depth, there is no dominating class. For this, we will utilize the following inequality version of Stirling's approximation, due to Robbins \cite{stirlingappr}.

\begin{proposition}[Stirling's approximation \cite{stirlingappr}]
    For all $n \in \Nset$, $n > 0$, we have
    \[
    \sqrt{2\pi n} \Big(\frac{n}{e}\Big)^n e^{\frac{1}{12n+1}} < n! < \sqrt{2\pi n} \Big(\frac{n}{e}\Big)^n e^{\frac{1}{12n}}
    \]
\end{proposition}

\begin{lemma}\label{lemma:no-dominating-class}
    For any $f(n) = o(\sqrt{n})$, if the counting depth is $d(n) \geq n/t - f(n)$, then $\equiv_{d(n)}$ has no dominating class as $n \to \infty$.
\end{lemma}
\begin{proof} 
Let $M(n)$ be an arbitrary sequence of classes. We will show that there is some $n_0$ such that for every $n \geq n_0$ the probability of a model belonging to class $M(n)$ is at most half.

Let $M_{\overline{n}}$ be a class of the sequence $M(n)$ with $i, j \leq t$ such that $n_i \neq n_j$. Let $\overline{n}'$ be obtained from $\overline{n}$ by switching the numbers $n_i$ and $n_j$ in the tuple. We clearly have $|M_{\overline{n}'}| = |M_{\overline{n}}|$ by symmetry so the probability of a model belonging to $M_{\overline{n}}$ is at most half.

It remains to show that classes of the sequence $M(n)$ with tuples repeating only one number have probability at most half, when $n$ is large enough. Let $d \geq n/t - f(n)$, where $f(n) = o(\sqrt{n})$. Assume $d < n/t$ and let $M = M_{\overline{n}}$, where $\overline{n} = (d, \dots, d)$. Note that this is the only $(n, d)$-admissible tuple that repeats only one number. We show that this class $M$ has probability less than half if $n$ is large enough. In fact, we establish the stronger claim that the sequence of such classes with tuples $(d(n), \dots, d(n))$ has limit probability~$0$.

The size of the above class $M$ is given by the sum
\[
|M| = \sum\limits_{\substack{n_1+\cdots + n_t = n \\ n_i \geq d}} \binom{n}{n_1, \dots, n_t}.
\]

First we note that a multinomial coefficient is largest, when the numbers $n_1, \dots, n_t$ are equal. Using this and Stirling's approximation, we get

\begin{align*}
    \binom{n}{n_1, \dots, n_t} &\leq \binom{n}{\frac{n}{t}, \dots, \frac{n}{t}} \\ &\leq \frac{\sqrt{2\pi n} (\frac{n}{e})^n e^{\frac{1}{12n}}}{(\sqrt{2\pi \frac{n}{t}} (\frac{n}{et})^{\frac{n}{t}} e^{\frac{1}{12\frac{n}{t}+1}})^t} \\
    &= \frac{\sqrt{2\pi n} (\frac{n}{e})^n e^{\frac{1}{12n}}}{\sqrt{2\pi \frac{n}{t}}^t (\frac{n}{e})^{n} \frac{1}{t^n} e^{\frac{t}{12\frac{n}{t}+1}}} \\
    &= \sqrt{\frac{t^t}{(2\pi)^{t-1}}} \cdot \frac{1}{\sqrt{n^{t-1}}} \cdot e^{g(n)} \cdot t^n \\
    &\leq \sqrt{\frac{t^t}{(2\pi)^{t-1}}} \cdot \frac{1}{\sqrt{n^{t-1}}} \cdot 2^{n|\tau|}
\end{align*}

The exponent of $e$ above is 
\begin{align*}
g(n) &= \frac{1}{12n}-\frac{t}{12\frac{n}{t}+1} = \frac{1}{12n}-\frac{t^2}{12n+t} \\
&= \frac{12n+t -12t^2n}{12n(12n+t)} = \frac{(1-t^2)12n +t}{12n(12n+t)}.
\end{align*}
Clearly $g(n) < 0$ for all positive $n$ so $e^{g(n)} < 1$ and the above estimate holds.

By the stars and bars method, the original sum has 
\[
\binom{tf(n)+t-1}{t-1} \leq (2t)^{t-1} \cdot f(n)^{t-1}
\]
terms so the final estimate is
\[
|M| \leq \sqrt{\frac{t^t}{(2\pi)^{t-1}}} \cdot (2t)^{t-1} \cdot \frac{f(n)^{t-1}}{\sqrt{n^{t-1}}} \cdot 2^{n|\tau|}
\]

Recall that there are $2^{n|\tau|}$ models of size $n$ in total and $f(n) = o(\sqrt{n})$. We obtain the following limit:
\[
\lim_{n \to \infty} \frac{\sqrt{\frac{t^t}{(2\pi)^{t-1}}} \cdot (2t)^{t-1} \cdot \frac{f(n)^{t-1}}{\sqrt{n^{t-1}}} \cdot 2^{n|\tau|}}{2^{n|\tau|}} = 0.
\]
We see that the sequence of classes $M_{\overline{n}}$ with $\overline{n} = (d, \dots, d)$ has limit probability $0$. Thus for a single such class, the probability is certainly at most half if $n$ is large enough.

Now assume $d \geq n/t$, and consider the sequence of classes $M_{\overline{n}}$, where $\overline{n} = (n/t, \dots, n/t)$. This is again the only $(n, d)$-admissible tuple that repeats only one number. The size of these classes is given by $\binom{n}{\frac{n}{t}, \dots, \frac{n}{t}}$ and it is easy to see from the above that this sequence also has limit probability 0. Thus a single such class has probability less than half for large enough $n$.
\end{proof}

Intuitively, a class (sequence) is \emph{vanishing}, if a random model of size $n$ belongs to it with limit probability zero. To define this notion formally, fix a sequence $\equiv_{d(n)}$. We say that all classes in $\equiv_{d(n)}$ are \textbf{vanishing} as $n\rightarrow \infty$ if for all class sequences $M(n)$ for $\equiv_{d(n)}$, we
have $p_n\rightarrow 0$ as $n\rightarrow \infty$. 
We now show that if the counting depth is high enough, then all classes in the distribution sequence are vanishing. 

\begin{lemma}
    For any $f(n)$ such that $f(n) = \omega(\sqrt{n})$, if the counting depth is $d(n) \geq n/t + f(n)$, then all classes in $\equiv_{d(n)}$ are vanishing as $n \to \infty$.
\end{lemma}
\begin{proof}
    Fix some function $f(n)$ such that $f(n) = \omega(\sqrt{n})$. We first show that with limit probability zero a random model realizes each type less than $n/t + f(n)$ times. Without loss of generality we can assume that $f(n) \leq n - n/t$. Given a type $\pi$, we let $X_\pi$ denote the same random variable as in the proof of Lemma \ref{lemma:dominating-class}. Chernoff bound give us again that for every $n$ and for every $0 < \delta \leq 1$ we have 
    \[\Pr\bigg[X_\pi \geq (1 + \delta)\frac{n}{t}\bigg] \leq e^{-\delta^2 \frac{n}{t(2 + \delta)}} \leq e^{-\delta^2 \frac{n}{3t}}\]
    Note that $\delta$ can depend on $n$. Setting $\delta(n) := \sqrt{3tg(n)/n}$, for any $g(n)$ such that $g(n) = \omega(1)$ and $g(n) \leq n/(3t)$ (to guarantee that $\delta(n) \leq 1$), we get that
    \[e^{-\delta^2 \frac{n}{3t}} = e^{-g(n)} \to 0.\]
    Furthermore 
    \[\delta(n) \frac{n}{t} = \underbrace{\sqrt{3t} \cdot \frac{1}{t}}_{=: C} \cdot \sqrt{g(n)} \cdot \sqrt{n} = C \sqrt{ng(n)}.\]
    Hence for any function $g(n)$ such that $g(n) = \omega(1)$ and $g(n) \leq n/(3t)$ we have that with high probability the type $\pi$ is realized less than $\frac{n}{t} - C \sqrt{n g(n)}$ many times. Since $\pi$ was arbitrary, it follows from the union bound that with high probability every type $\pi$ is realized less than $\frac{n}{t} - C \sqrt{n g(n)}$ many times. By setting $g(n) = (f(n)/C \sqrt{n})^2$ we obtain that every type is realized less than $n/t - f(n)$ times.

    Now consider an arbitrary class sequence $M(n)$. We want to show that for every $\varepsilon > 0$ we have that $p_n < \varepsilon$, provided that $n$ is sufficiently large. Consider a class $M(n)$ and let $\overline{n}$ be the corresponding $(n,d(n))$-admissible tuple, i.e., $M(n)$ is the class $M_{\overline{n}}$. If there is $i \in I$ such that $n_i = d(n) = n/t + f(n)$, then it follows from the previous result that $p_{\equiv_{d(n)}}(M_{\overline{n}}) < \varepsilon$, as long as $n$ is sufficiently large. Suppose then that $n_i < n/t + f(n)$, for every $i \in I$. In this case the class $M_{\overline{n}}$ is an isomorphism class, which means that the probability that a random model belongs to $M_{\overline{n}}$ is simply
    \[\binom{n}{n_1,\dots,n_t}/2^{n|\tau|},\]
    which, as calculated in the proof of Lemma \ref{lemma:no-dominating-class}, is at most a constant times $1/\sqrt{n^{t-1}}$. This latter quantity is certainly less than $\varepsilon$, provided that $n$ is sufficiently large. Hence $p_{\equiv_{d(n)}}(M_{\overline{n}}) < \varepsilon$, provided that $n$ is sufficiently large.
\end{proof}

We gather the above results in the following theorem:
\begin{theorem}\label{thm:qualitative-theorem}
The following statements hold
for counting depth $d(n)$ as $n \to \infty$.
\begin{itemize}
\item
If $d(n) \leq n/t - f(n)$
where $f(n) = \omega(\sqrt{n})$,
then $\equiv_{{d(n)}}$ has a dominating class. 
\item
If $d(n) \geq n/t - f(n)$ where $f(n) = o(\sqrt{n})$, then $\equiv_{{d(n)}}$ has no dominating class.
\item
If $d(n) \geq n/t + f(n)$ where $f(n) = \omega(\sqrt{n})$, then every class in $\equiv_{d(n)}$ is vanishing.
\end{itemize}
\end{theorem}

We point out a corollary of the above result. When the counting depth is too low, almost all models are in the same dominating class in terms of $\gmlu_d$ definability. Conversely, if the counting depth is high enough, $\gmlu_d$ can separate the models into classes that vanish. These observations directly give us the following result:
\begin{corollary}\label{cor:hassu-korollaari}
    Let $f(n) = \omega(\sqrt{n})$. If $d(n) \leq n/t - f(n)$, then with limit probability one, two random models of size $n$ cannot be separated in $\gmlu_{d(n)}$. If $d(n) \geq n/t + f(n)$, then with limit probability one, two random models of size $n$ can be separated in $\gmlu_{d(n)}$.
\end{corollary}

We say that a class is a \textbf{majority class}, if it contains more than half of all the models. The following theorem is a quantitative version of Theorem \ref{thm:qualitative-theorem}.


\begin{theorem}
    Let $n \in \mathbb{Z}_+$.
    \begin{itemize}
        \item If $d \leq n/t - c_1\sqrt{n}$, where
        \[c_1 := \frac{1}{t}\sqrt{2t\ln(2t)},\]
        then the distribution $\equiv_d$ for models of size $n$ has a majority class.
        \item  If $d \geq n/t- c_2 \sqrt{n}$, where
        \[
        c_2 := \sqrt{\frac{\pi}{2t^3(4t)^{1/(t-1)}}} < c_1
        \]
        then the distribution $\equiv_d$ for models of size $n$ does not have a majority class.
    \end{itemize}
\end{theorem}
\begin{proof}
    Suppose first that $d \leq n/t - c_1 \sqrt{n}$. Set $\delta = (tc_1)/\sqrt{n}$, in which case $\delta \cdot \frac{n}{t} = c_1\sqrt{n}$ and $\delta^2 \cdot n/(2t) = (tc_1)^2/(2t)$. Applying Chernoff bound and the union bound we obtain, in a similar manner as in the proof of Lemma \ref{lemma:dominating-class}, that with probability strictly greater than $(1 - te^{-(tc_1)^2/(2t)})$ every type is realized at least $n/2 - c_1 \sqrt{n}$-times. A quick calculation shows that this latter probability is equal to $1/2$ (hence the choice of $c_1$).

    Consider then the case $d \geq n/t-c_2\sqrt{n}$. Let $M = M_{\overline{n}}$, where $\overline{n} = (d, \dots, d)$. Using the estimate from the proof of Lemma \ref{lemma:no-dominating-class} with $f(n) = c_2\sqrt{n}$, we obtain
    \begin{align*}
    \frac{|M|}{2^{n|\tau|}} &\leq \sqrt{\frac{t^t}{(2\pi)^{t-1}}} \cdot (2t)^{t-1} \cdot \frac{(c_2\sqrt{n})^{t-1}}{\sqrt{n^{t-1}}} = 1/2
    \end{align*}
    Thus $M$ is not a majority class. The same reasoning as in the proof of Lemma \ref{lemma:no-dominating-class} shows there is no other majority class.
\end{proof}

We of course obtain a corresponding quantitative version of Corollary \ref{cor:hassu-korollaari}.

\begin{corollary}
    Let $n \in \mathbb{Z}_+$. If $d \leq n/t - c_1 \sqrt{n}$, then the probability that two random models of size $n$ can be separated in $\gmlu_d$ is less than $1/4$. If $d \geq n/t - c_2\sqrt{n}$, then the probability that two random models of size $n$ can be separated in $\gmlu_d$ is at least $1/2$.
\end{corollary}

\section{Conclusion}\label{conclusion}
We have established an interesting monotone connection between model class sizes and description complexities, also obtaining related results for entropy. Furthermore, we have characterized the phase transitions of model class size when the domain size $n$ and expressive power (in the form of counting depth $d$) is altered. These results elucidate the interplay of class size and formula length. While focusing on $\gmlu_d$, the results have been intended to give a general overview of related phenomena. Thereby, an obvious future direction involves investigating how these results lift into the framework of first-order logic. There, natural parametrizations---analogous to varying the counting depth $d$---can possibly be obtained by using quantifier depth and the number of variables.

Moving beyond first-order logic, it would be interesting to investigate the expressively Turing-complete logic $\mathrm{CL}$, or \emph{computation logic}, introduced in \cite{turingcomp}. Studying description complexities within that framework would lead to an even closer link to Kolmogorov complexity.

\medskip

\medskip

\medskip

\noindent
\textbf{Acknowledgments.}
Antti Kuusisto and Miikka Vilander were supported by the Academy of Finland project \emph{Explaining AI via Logic} (XAILOG), grant number 345612 (Kuusisto). Antti Kuusisto was also supported by the Academy of Finland project \emph{Theory of computational logics}, grant numbers 324435, 328987, 352419, 352420, 352419, 353027.

\bibliographystyle{plain}
\bibliography{lics}

\end{document}